\documentclass{amsart}
\usepackage{amsmath,amssymb,amsthm,amsfonts,mathrsfs}
\usepackage{mathtools}
\usepackage{dsfont}
\usepackage{fullpage}
\usepackage{color}
\usepackage[colorlinks=true,
linkcolor=blue,
anchorcolor=blue,
citecolor=red
]{hyperref}

\newtheorem{theorem}{Theorem}[section]
\newtheorem{corollary}[theorem]{Corollary}
\newtheorem{lemma}[theorem]{Lemma}
\newtheorem{proposition}[theorem]{Proposition}
\theoremstyle{definition}
\newtheorem{definition}[theorem]{Definition}
\theoremstyle{remark}
\newtheorem{remark}[theorem]{Remark}

\newcommand{\N}{\mathbb{N}}
\newcommand{\Z}{\mathbb{Z}}
\newcommand{\Q}{\mathbb{Q}}

\newcommand{\C}{\mathbb{C}}
\newcommand{\F}{\mathbb{F}}

\renewcommand{\P}{\mathbb{P}}

\newcommand{\cO}{\mathcal{O}}
\newcommand{\cS}{\mathcal{S}}

\newcommand{\cP}{\mathcal{P}}

\renewcommand{\mod}[1]{\,(\on{mod}#1)}
\newcommand{\of}[1]{\left(#1\right)}
\newcommand{\set}[1]{\left\{#1\right\}}
\newcommand{\abs}[1]{\left\vert#1\right\vert}

\newcommand{\on}{\operatorname}
\renewcommand{\Re}{\on{Re}}

\newcommand{\Gal}{\on{Gal}}

\newcommand{\cD}{\mathcal{D}}
\newcommand{\cK}{\mathcal{K}}
\newcommand{\fp}{\mathfrak{p}}
\newcommand{\fP}{\mathfrak{P}}

\newcommand{\RNum}[1]{\uppercase\expandafter{\romannumeral #1\relax}}

\title[Alladi's formula]{Analogues of Alladi's formula over global function fields}
\author[L. Duan]{Lian Duan}
\address{ Department of Mathematics, Colorado State University, Fort Collins, Colorado 80523, USA}
\email{lian.duan@colostate.edu}

\author[B. Wang]{Biao Wang}
\address{Department of Mathematics, State University of New York at Buffalo, Buffalo, NY 14260, USA}
\email{bwang32@buffalo.edu}

\author[S. Yi]{Shaoyun Yi}
\address{Department of Mathematics, University of South Carolina, Columbia, SC 29208, USA}
\email{shaoyun@mailbox.sc.edu}

\date{\today}
\makeatletter
\@namedef{subjclassname@2020}{\textup{}2020 Mathematics Subject Classification}
\makeatother

\begin{document}
	\begin{abstract}
In this paper, we show an analogue of Kural, McDonald and Sah's result on Alladi's formula for global function fields. Explicitly, we show that for a global function field $K$,  if a set $S$ of prime divisors has a natural density $\delta(S)$ within prime divisors, then 
$$-\lim_{n\to\infty}  \sum_{\substack{1\le \deg D\le n\\ D\in \mathfrak{D}(K,S)}}\frac{\mu(D)}{|D|}=\delta(S),$$
where $\mu(D)$ is the M\"{o}bius function on divisors and $\mathfrak{D}(K,S)$ is the set of all effective distinguishable divisors whose smallest prime factors are in $S$. As applications, we get the analogue of Dawsey's and Sweeting and Woo's results to the Chebotarev Density Theorem for function fields, and the analogue of Alladi's result to the Prime Polynomial Theorem for arithmetic progressions.  We also display a connection between the M\"obius function and the Fourier coefficients of modular form associated to elliptic curves. The proof of our main theorem is similar to the approach in  Kural et al.'s  article. 
	\end{abstract}
	
\maketitle
	
\section{Introduction and statement of results}\label{Sect: intro_and_main_thm}

Let $n\ge1$ be an integer. Let $\mu(n)$ be  the M\"{o}bius function defined by $\mu(n)=(-1)^k$ if $n$ is the product of $k$ distinct primes and zero otherwise.  
Let $p_{\min}(n)$ be the smallest prime factor of $n$ and let $p_{\max}(n)$ be the largest prime factor of $n$. Let $p_{\min}(1)=p_{\max}(1)=1$. In 1977, Alladi  \cite{Alladi1977duality}  introduced a duality between $p_{\min}(n)$ and  $p_{\max}(n)$, which says that for any function $f$ defined on integers  with $f(1)=0$, we have
\begin{align}
	\sum_{d|n}\mu(d)f(p_{\min}(d))&=-f(p_{\max}(n)),\label{duality1}\\
	\sum_{d|n}\mu(d)f(p_{\max}(d))&=-f(p_{\min}(n)).\label{duality2}
\end{align}
Applying the prime number theorem for arithmetic progressions and \eqref{duality1}, he \cite{Alladi1977} showed that if $(\ell,k)=1$, then
\begin{equation}\label{alladi}
	-\sum_{\substack{n\geq 2\\ p_{\min}(n)\equiv \ell (\on{mod}k)}}\frac{\mu(n)}{n}=\frac1{\varphi(k)},
\end{equation}
where $\varphi$ is Euler's totient function. 

Alladi's formula \eqref{alladi} shows a relationship between the M\"{o}bius function $\mu(n)$ and the density of primes in arithmetic progressions.   In 2017, Dawsey \cite{Dawsey2017} first generalized \eqref{alladi} to the setting of Chebotarev densities for finite Galois extensions of $\Q$.  Then Sweeting and Woo \cite{SweetingWoo2019}  generalized   Dawsey's result to number fields.  Recently, Kural, McDonald and Sah \cite{KuralMcDonaldSah2020}  generalized  all these results to natural densities of sets of primes. The second author of this article  showed the analogues of these results over $\Q$ for some arithmetic functions other than $\mu$ in \cite{Wang2020ijnt, Wang2020rj, Wang2020jnt}. We note here that Ono, Schneider and Wagner showed some beautiful partition-theoretic analogues of Alladi's formula \eqref{alladi} in \cite{OnoSchneiderWagner2017, OnoSchneiderWagner2021}. In this paper, we will show the analogue of Kural et al.'s result \cite{KuralMcDonaldSah2020}  over global function fields.

Let $p$ be a prime and let $q$ be a power of prime $p$. Let $\F_q$ be a finite field of $q$ elements. 
Take $K/\F_p(x)$ to be a finite extension with constant field $\F_q$, which is called a global function field (or simply a function field). A \emph{prime divisor} (or simply a \emph{prime}) in $K$ is defined to be a discrete valuation ring $R_P$ with maximal ideal $P$ such that $\F_q\subseteq R_P$ and the quotient field of $R_P$ is $K$. The \emph{norm} of $P$, denote by $|P|$, is defined to be the size of the residue field $\kappa_P$ of $R_P$, i.e. $|P|=\#(R_P/P)=\#\kappa_P$, which is a power $q^{\deg P}$ of the cardinality of the ground field $\F_q$. Here the exponent $\deg P$ is called the \emph{degree} of $P$.  

A \emph{divisor} $D$ is a finite formal sum of prime divisors, i.e. 
\begin{equation}\label{Eqn: divisor}
	D=\sum_{P} a_P\cdot P, 
\end{equation}
such that every $a_P$ is an integer, and $a_P=0$ for all but finitely many $P$. If $D$ satisfies additionally that $a_P\geq 0$ for all $P$, then we say $D$ is \emph{effective} and write $D\geq 0$. We say $D$ is \emph{supported} by $P$, or $P$ is a \emph{prime factor} of $D$, and write $P|D$, if $a_P$ in \eqref{Eqn: divisor} is nonvanished. The \emph{degree} $\deg D$ of $D$ is defined by $\deg D:=\sum a_P \deg P$, and the norm of $D$ is $q^{\deg D}$. In particular, for every effective $D$, $\deg D\geq 0$, and hence $|D|\geq 1$. 

Let $\cD$ be the set consisting of all the divisors of $K$. Then $\cD$ has two natural subsets: one is the subset of all effective divisors, denoted by $\cD^+$; and the other one is the subset which  consists of all prime divisors, denoted by $\cP$. For any subset $S\subseteq\cP$, we say $S$ has a \textit{natural density} $\delta(S)$, if the following limit exists:
\begin{equation}
	\delta(S):=\lim_{n\to\infty}\frac{\pi_{K,S}(n)}{\pi_K(n)},
\end{equation}
where $\pi_{K,S}(n):=\#\set{P\in S:\deg P=n}$, and $\pi_K(n):=\pi_{K,\cP}(n)$.

Let $d_{-}(D):=\min\set{\deg P: P|D}$ be the minimal degree of the prime factors of $D$. We say that $D$ is \textit{distinguishable} if $D\neq 0$ and there is a unique prime factor, say $P_{\min}(D)$, of $D$ attaining
the minimal degree $d_{-}(D)$. We define
\begin{equation}\label{Definition of distingushable}
	\mathfrak{D}(K,S):=\set{D\in \cD^+: D \text{ is distinguishable and } P_{\min}(D)\in S}.
\end{equation}
Analogous to the classical definition, if $D=0$ or $D=a_1P_1+\cdots+a_kP_k$ with all $a_i>0$, the M\"obius function $\mu(D)$ is defined by 
$$
\mu(D)=
\begin{cases}
	1 & \text{ if }D=0,\\
	(-1)^k & \text{ if }k>0 \text{ and }a_i=1 \text{ for all }i,\\
	0 & \text{ if at least one }a_i>1.
\end{cases}
$$ 

The following theorem is our main result.

\begin{theorem}\label{mainthm}
	Given a global function field $K$, with the above notations, if $S\subseteq\cP$ has a natural density $\delta(S)$, then we have that
	\begin{equation}\label{mainthmeq}
		-\lim_{n\to\infty}  \sum_{\substack{1\le \deg D\le n\\ D\in \mathfrak{D}(K,S)}}\frac{\mu(D)}{|D|}=\delta(S).
	\end{equation}
\end{theorem}

We want to take the rest of this section to introduce several consequences of our main result. The first one is an application of the Chebotarev Density Theorem of function field (cf. Theorem~\ref{Thm: Chebotarev_A}). Recall that a finite extension $L/K$ of function field is called \emph{geometric} if $L$ and $K$ have the same constant field. In this case, for every unramified prime divisor $P$ of $K$, we can associate it with a conjugacy class of $\Gal(L/K)$, called the \emph{Frobenius class} of $P$ and denoted by $(P, L/K)$. For more details about this part, see Sect.~\ref{Sect: Gal_ext}. 
\begin{corollary} \label{cordawsey_AFF}
	Let $L/K$ be a geometric Galois extension with Galois group $G = \Gal(L/K)$. Then for any conjugacy class $C\subseteq G$, we have
	\begin{equation}\label{cordawseyeq_AFF}
		-\lim_{n\to\infty} \sum_{\substack{D \in \mathfrak{D}(K,\cP), 1\le \deg D\le n \\ (P_{\min}(D), L/K)=C }}\frac{\mu(D)}{|D|}=\frac{\#C}{\#G}.
	\end{equation}	
\end{corollary}

\begin{proof}
	Suppose that the common constant field of $L$ and $K$ is $\F_q$.    For any conjugacy class $C\subseteq G$, define
	$$S_C:=\set{P\in\cP: (P, L/K)=C}.$$
	It follows from the classical Chebotarev Density Theorem for function fields (see Theorem~\ref{Thm: Chebotarev_B}) that 
	\begin{equation*}
		\abs{\pi_{K,S_C}(n)-\frac{\#C}{\#G}\frac{q^n}{n}}\ll\frac{q^{n/2}}{n}.
	\end{equation*}
	Since $\pi_K(n)\sim q^n/n$  by the Prime Number Theorem for function fields (see Lemma~\ref{Lem: prime_num_thm}), taking $S=S_C$ in Theorem~\ref{mainthm} gives $\delta(S)=\#C/\#G$, and \eqref{cordawseyeq_AFF} follows. 
\end{proof}

Now we specialize to the rational function field over $\F_q$. Fix the choice of the parameter $x$ in $K$, one can realize $K$ as $\F_q(x)$. By the well-known relationship between the projective line $\P^1_{\F_q}$ and its affine chart $\mathbb{A}^1_{\F_q, x}=\P^1_{\F_q}-\{\infty\}$, for a polynomial $F$ in the affine coordinate ring $\F_q[x]$, we take the \emph{divisor of zeros} of $F$ to be
$$(F)_0:=\sum_{{\rm{ord}}_P(F)\geq 0} {\rm{ord}}_P(F)P.$$
Using the fact that $\F_q(x)$ has class number one, the map $\phi: F\mapsto (F)_0$ gives a bijection of the following sets. 
\begin{align*}
	\set{\text{monic polynomials $F\in \F_q[x]$}} & \Longleftrightarrow \set{\text{effective divisors $D$ not supported by $(\infty)$}}\\
	\set{\text{irreducible monic polynomials $F_P\in \F_q[x]$}} & \Longleftrightarrow \set{\text{prime divisors $P\neq (\infty)$}}
\end{align*}
For the definition of class number, see Sect.~\ref{Sect: genus}.  One can verify that $\phi$ defined above is with respect to the degree, norm and prime factorization. In particular, if we let $d_{-}(F):=\min \set{ \deg F_P: F_P|F, F_P\ \text{is irreducible}}$ then 
$
d_{-}(F)=d_{-}(\phi(F))
$. Hence the distinguishable elements are mapped to the distinguishable elements. 
In addition, let $S$ be a subset of monic irreducible polynomials of $\F_q[x]$, we have 
$$
\pi_{q, S}(n):=\#\set{F_P\in S: \deg F_P=n}=\pi_{\F_q(x), \phi(S)}(n).
$$
Thus $\delta(S)=\delta(\phi(S))$. For a distinguishable polynomial $F$, we let $p_{\min}(F)$ denote the minimal prime factor of $F$, and let 
\begin{equation*}
	\mathfrak{D}(q,S):=\set{F\in\F_q[x]: F  \text{ is monic and distinguishable, and } p_{\min}(F)\in S}.
\end{equation*}
Then 
$$
\phi(\mathfrak{D}(q,S))=\mathfrak{D}(\F_q(x), \phi(S)). 
$$
With all the above, we are ready to state the application of Theorem~\ref{mainthm} to $\F_q[x]$. 

\begin{corollary}\label{mainthm_rational}
	Let  $q\ge2$ be fixed.	If $S\subseteq\cP$ has a natural density $\delta(S)$, then we have that
	\begin{equation}
		-\lim_{n\to\infty}  \sum_{\substack{1\le \deg F\le n\\ F\in \mathfrak{D}(q,S)}}\frac{\mu(F)}{|F|}=\delta(S).
	\end{equation}
\end{corollary}
\begin{proof}
	This corollary directly follows from the specialization and the correspondence above. Note that the set $\{(\infty)\}$ has density zero. 
\end{proof}

As another application, we get the analogue of  Alladi's formula over finite fields  by the Prime Polynomial Theorem for arithmetic progressions.  

\begin{corollary}\label{coralladi}
	For relatively prime $f, g\in\F_q[x]$, we have
	\begin{equation}\label{coralladieq}
		-\lim_{n\to\infty}  \sum_{\substack{F \in \mathfrak{D}(q,\cP), 1\le \deg F\le n\\ p_{\min}(F)\equiv f\mod g}}\frac{\mu(F)}{|F|}=\frac1{\varphi(g)}.
	\end{equation}	
	Here $\varphi(g)$ is the function field Euler totient function given by the number of units in $\F_q[x]/g\F_q[x]$.
\end{corollary}

\begin{proof}
	For relatively prime $f,g\in \F_q[x]$, let 
	$$S(f, g)=\set{p\in \cP: p=f+gh \text{ for some } h\in \F_q[x]}.$$
	Then 
	the Prime Polynomial Theorem for arithmetic progressions (e.g., see \cite[Theorem 4.8]{Rosen2002}) says that
	\begin{equation}
		\pi_{q, S(f,g)}(n)=\frac{\pi_q(n)}{\varphi(g)}+O\of{\frac{q^{n/2}}{n}}.
	\end{equation}
	
	Taking $S=S(f, g)$ in Theorem~\ref{mainthm} gives
	$\delta(S)=1/\varphi(g)$, and (\ref{coralladieq}) follows.
\end{proof}

\begin{remark}
	In fact, one can verify that the above corollary is a special case of Corollary~\ref{cordawsey_AFF}.
\end{remark}

\subsection{Notations} We use the notations $f(n)=O_{K,\theta,\dots}(g(n))$ or $f(n)\ll_{K,\theta,\dots}g(n)$ to mean that there is a positive constant $C=C(K,\theta,\dots)>0$ such that $|f(n)|\le C|g(n)|$ for all $n\ge1$, and $C$ is called the implied constant of this $O$-term. Sometimes we  write $f=O(g)$ or $f\ll g$ for simplicity, if the implied constant depends at most on the field $K$ and some fixed positives. We use capital letters $A,B,D$ to denote divisors. Our main results and their proofs only concern effective divisors, hence all the divisors after Sect.~\ref{Sect: Background} will be assumed to be effective. The
following table offers an analogy between some of our notations on function fields and Kural et. al's notations on number fields.
\begin{center}
	\begin{tabular}{|c|c|} \hline 
		Function fields notations & Number fields notations \\ \hline 
		Function field $K/\mathbb{F}_p(x)$    & Number field $K/\mathbb{Q}$ \\
		$D\in\mathcal{D}^+$&$\mathfrak{a}\subseteq \mathcal{O}_K$\\
		$D=\sum\limits_Pa_P\cdot P$&$\mathfrak{a}=\prod\limits_{i=1}^r\mathfrak{p}_i^{e_i}$\\
		$\mu(D)$&$\mu(\mathfrak{a})$\\
		$|D|=q^{\deg D}$&$\mathrm{N}(\mathfrak{a})$\\
		$\mathfrak{D}(K,S)$&$D(K,S)$\\
		$d^+(D)$&$M(\mathfrak{a})$\\
		$Q_S(D)$&$Q_S(\mathfrak{a})$\\
		$\pi_{K,S}(n)$&$\pi_S(K;X)$ \\ \hline 
	\end{tabular}   
\end{center}

\subsection{Organization of this paper}
Theorem~\ref{mainthm} may be viewed as a discrete version of Kural et al.'s result in \cite[Theorem 1.1]{KuralMcDonaldSah2020}. Hence the approach to proving Theorem~\ref{mainthm} here is similar to  their proof. After reviewing some concepts and theorems related to  the main results in Sect.~\ref{Sect: Background}, we develop a duality identity between the prime factors of divisors in Sect.~\ref{Sect: duality}. Then in Sect.~\ref{Sec:largest_prime_divisor}, we are devoted to finding the asymptotic formula for the number of prime factors of divisors that attain the largest degree. Finally, in Sect.~\ref{Sec:intermediate_thm}, we show the analogue of Alladi's intermediate result \cite[Theorem 6]{Alladi1977} that derives the desired formula \eqref{mainthmeq} from the  asymptotic result in Sect.~\ref{Sec:largest_prime_divisor}. 
A feature in Sect.~\ref{Sec:intermediate_thm} is that for function fields we can show the analogue of \cite[Theorem 6]{Alladi1977} without using  Axer's theorem, which was applied  in \cite{Alladi1977,  KuralMcDonaldSah2020}. Finally in Sect.~\ref{Sect: modular_explanation}, when the function fields are coming from the reduction of an elliptic curve $E$ over $\Q$, we study the connection between the restricted sums of $\mu(D)/|D|$ and the Fourier coefficients of the modular form corresponding to this elliptic curve. 

\section{Background}\label{Sect: Background}
In this section, we will review some concepts and theorems, including the Prime Number Theorem  and the Chebotarev Density Theorem for function fields. In particular, the class number and genus will be reviewed in Sect.~\ref{Sect: genus}. In Sect.~\ref{Sect: Gal_ext} we recall the necessary proprieties of the geometric extensions of algebraic function fields. The main references for this section are \cite[Chapters~5, 7, 8, 9]{Rosen2002}, \cite[Chapter~8]{Fulton-Alg-Curve}  and \cite{Atiyah-Macdonald-commutative-alg}. Since there is nothing original in this section, the experts of the function fields are suggested to directly move to Sect.~\ref{Sect: duality}. 

\subsection{Prime Number Theorem for function fields }\label{Sect: prime_num_thm}
It is well known that the Prime Number Theorem (PNT) says that the number $\pi(x)$ of primes up to $x$ is asymptotic to ${x}/{\log x}$ as $x\to\infty$, and the Riemann Hypothesis is equivalent the assertion that $\pi(x)=\on{Li}(x)+O\of{x^{1/2}\log x}$ (e.g., see \cite[Theorem 6.1]{Koukoulopoulos2020}), where $\on{Li}(x)=\int_2^xdt/\log t$.  In the function fields,  for the number $\pi_K(n)$ of prime divisors with degree $n$,  we have the following analogue (e.g. see \cite[Theorem~5.12]{Rosen2002}). 

\begin{lemma}[PNT for function fields]\label{Lem: prime_num_thm}
	We have
	\begin{equation}\label{Lem: prime_num_thm_eq}
		\pi_K(n)=\frac{q^n}{n}+O\of{\frac{q^{n/2}}{n}}.
	\end{equation}
\end{lemma}
This is  analogous to the asymptotic formula $\pi(x)=\on{Li}(x)+O\of{x^{1/2}(\log x)^2}$,
if one takes $q^n$ as $x$ and $n$ as $\log x$. And we will use the asymptotic  $\pi_K(n)\sim{q^n}/{n}$ as $n\to\infty$ several times in the later proofs.

\subsection{Class number and genus}\label{Sect: genus}
In this section, we will fix a function field $K$ and will follow the notations introduced in Sect.~\ref{Sect: intro_and_main_thm}. Recall that the divisor group $\cD$ is an abelian group with identity the zero divisor $0$, and so the inverse of $D=\sum a_P P$ is $-D=\sum -a_P P$. If $x\in K$, then the \emph{principal divisor} $(x)$ is defined by
$$
(x):=\sum {\rm{ord}}_P(x) P.
$$
Here ${\rm{ord}}_{P}: K\to \Z$ is the order function attached to $P$; see \cite[Chapter~9]{Atiyah-Macdonald-commutative-alg}. Note that $R_P$ is a discrete valuation ring. In particular, it follows from \cite[Proposition~5.1]{Rosen2002} that $\deg(x)=0$. 

We say two divisors $D_1$ and $D_2$ are \emph{linearly equivalent}, $D_1\sim D_2$ if $D_1-D_2=(x)$ for some $x\in K$. It is not hard to verify that the set $\mathcal{P}r:=\{(x): x\in K\}$ is a subgroup of $\cD$. Thus there is a natural bijection between the linearly equivalent classes of divisors and the quotient group $\cD/\mathcal{P}r$. Let $\mathcal{C}l_K:=\cD/\mathcal{P}r$. Then $\mathcal{C}l_K$ has a group structure induced by that of $\cD$. Moreover, there is a natural group homomorphism $\deg: \mathcal{C}l_K\to \Z, [D]\mapsto \deg D$, which sends every equivalent class $[D]$ of $D$ to $\deg D$. Note that this map is well-defined since all principal divisors have degree $0$. We take $\mathcal{C}l_K^0:=\ker \deg$, this subgroup of $\mathcal{C}l_K$ is called the \emph{class group} of $K$. It is a fact that $\mathcal{C}l_K^0$ is a finite group (\cite[Lemma~5.6]{Rosen2002}) and its size $h_K$ is called the \emph{class number} of $K$. 

Given a divisor $D$, we define its corresponding  \emph{Riemann-Roch space} $$L(D):=\{x\in K^*: (x)+D\in \mathcal{D}^+\}\cup \{0\}.$$ One can check that if $x, y\in L(D)$, then so are $x+y$ and $\alpha x$ for any $\alpha\in \F_q$. Hence $L(D)$ is a vector space over $\F_q$. According to \cite[$\S$~8.2, Proposition~3.(3)]{Fulton-Alg-Curve},  $l(D):=\dim L(D)$ is finite. Moreover, if $\deg(D)\geq 0$, then $l(D)\leq \deg(D)+1$.

\begin{theorem}\cite[Theorem~5.4, Rimann-Roch theorem]{Rosen2002}\label{Thm: RR}
	There is an integer $g_K\geq 0$ and a \emph{canonical divisor class} $\mathcal{C}$ such that for any $C\in \mathcal{C}$ and $D\in \cD$ we have 
	$$
	l(D)=\deg(D)-g_K+1+l(C-D). 
	$$
\end{theorem}

\begin{remark}\label{Rmk: genus}
	The constant $g_K$ in the above theorem only depends on $K$. It is called the \emph{genus} of $K$. In fact, we have $l(C)=g_K$ and $\deg(C)=2g_K-2$ for any $C\in \mathcal{C}$ (see \cite[Corollaries of Theorem~5.4]{Rosen2002}).
\end{remark}

Here we cite a result which will be used in the other sections. 
\begin{proposition}\cite[Lemma~5.8]{Rosen2002}\label{Prop: formula_of_bn}
	For every integer $n$, there are $h_K$ divisor classes of degree $n$. Suppose $n\geq 0$ and that $\{A_1, \cdots, A_{h_K}\}$ are the representatives of the classes, then the number of effective divisors of degree $n$, $b_n$, is given by $\sum_{i=1}^{h_K}\frac{q^{l(A_i)}-1}{q-1}$. 
\end{proposition}

\begin{corollary}\label{Cor: bn_for_large_n}
	If $n>2g_K-2$, then $b_n=h_K\frac{q^{n-g_K+1}-1}{q-1}$. 
\end{corollary}
\begin{proof}
	Apply the Riemann-Roch theorem to the case $n=\deg A_i>2g_K-2$ for all $i=1, \cdots, h_K$. Note that $L(C-A_i)=0$ since $\deg(C-A_i)<0$ (see Remark~\ref{Rmk: genus}). This mean $l(C-A_i)=0$, and $l(A_i)=\deg A_i-g_K+1$ for all $i$. Then this corollary follows immediately. 
\end{proof}

\subsection{Galois extensions}\label{Sect: Gal_ext}
In this subsection, we consider the extensions of function fields. Let $K$ be a function field  with constant field $F$ and let $L/K$ be a finite extension. Take $E$ to be the constant field of $L$, it is easy to see that $E$ is a finite extension of $F$. If $E=F$, we call $L$ a \emph{geometric extension} of $K$. In this and the following sections, we will focus on the Galois extension $L/K$, i.e. on the case $\#\Gal(L/K):={\rm{Aut}}(L/K)=[L:K]$.

Until the end of this subsection, we will fix a Galois extension $L/K$ of degree $n$, with the Galois group $G:=\Gal(L/K)$. Since $K$ is a subfield of $L$, for every prime divisor $\fP\in \cD_L$ in the divisor group of $L$, it is a fact that $R_{\fP}\cap K=R_P$ for a unique $P\in \cD_K$.  In this case, we say that $\fP$ is \emph{lying above} (or simply \emph{above}) $P$ and write $\fP|P$. Conversely, given a prime divisor $P\in \cD_K$, there are finitely many distinct prime divisors $\fP_1, \cdots, \fP_r$ lying above $P$, in particular, $r\leq n$. Now taking $\fP|P$, one can verify that $R_P$ is a subring of $R_{\fP}$ via the embedding $K\hookrightarrow L$. There are two integers related. The first is the \emph{relative degree} $f:=f(\fP/P)$, which is defined to be the extension degree of the residue fields $\kappa_{\fP}/\kappa_P$. Here $\kappa_{\fP}$ (resp.~$\kappa_P$) is the residue field of $R_{\fP}$ (resp.~$R_P$). The second integer is the \emph{ramification index} $e:=e(\fP/P)$, which is defined to be ${\rm{ord}}_{\fP}(\tau_P)$. Here $\tau_P$ is the uniformizer of the discrete valuation ring $R_P$. In particular, we say that $P$ is \emph{unramified} if $e(\fP/P)=1$ for one (and hence for all, see Proposition~\ref{Prop: efr=n} below) $\fP|P$, otherwise, we say $P$ is \emph{ramified}. The following results are well known. 

\begin{proposition}\cite[Propositioin~9.3]{Rosen2002}\label{Prop: efr=n}
	When $L/K$ is Galois, $e(\fP/P), f(\fP/P)$ and the integer $r$ above only depend on $P$. Moreover, $efr=n$. 
\end{proposition}

\begin{proposition}\cite[Theorem~7.12]{Rosen2002}
	Given a Galois extension $L/K$, there is a \emph{different divisor} $D_{L/K}\in \cD_L$ such that $P$ is ramified in $L$ if and only if there is any (and hence every) $\fP|P$ such that $\fP|D_{L/K}$. In particular, there are only finitely many ramified prime divisors of $L/K$. 
\end{proposition}

Now we focus on the group structure of $G=\Gal(L/K)$. Still taking $\fP|P$, there is a subgroup  $G_{\fP}:=\{\sigma\in G: \sigma(\fP)=\fP\}$ consists of all Galois elements which preserve $\fP$, it is called the \emph{decomposition group} of $\fP$. With the above notation, we know that $\#G_{\fP}=ef$. Moreover, there is a naturally defined group homomorphism $G_{\fP}\to \Gal(\kappa_{\fP}/\kappa_P)$ induced its action on the residue field $\kappa_{\fP}$. The kernel of this homomorphism will be called the \emph{inertia subgroup} of $\fP$, and is denoted by $I_{\fP}$. In fact, $\#I_{\fP}=e$. Thus $P$ is unramified if and only if $I_{\fP}$ is trivial for one (and hence all) $\fP|P$. Moreover in this case, we have $G_{\fP}\simeq \Gal(\kappa_{\fP}/\kappa_P)$. For more details about the decomposition subgroup and inert subgroup, see for example \cite[Chapter~9]{Rosen2002}.

Assume $P$ is unramified and let $\fP|P$. As we have seen above, $G_{\fP}\simeq \Gal(\kappa_{\fP}/\kappa_P)$. Observe that $\kappa_{\fP}/\kappa_P$ is an extension of finite field, thus has cyclic Galois group with the generator $\eta_P$ such that $\eta_P(\alpha)=\alpha^{|P|}$ for all $\alpha\in \kappa_{\fP}$. We take the notation $(\fP, L/K)$ to represent the unique element in $G_{\fP}$ which is equal to $\eta_P$ via the isomorphism $G_{\fP}\simeq \Gal(\kappa_{\fP}/\kappa_P)$. Thus under this isomorphism, we have $(\fP, L/K)(\alpha)=\alpha^{|P|}$. This element is called the \emph{Frobenius element} of $\fP$. 

\begin{proposition}\cite[Propositions~9.2, 9.7]{Rosen2002}
	Let $\set{\fP_1, \cdots, \fP_r}$ be the set of prime divisors above $P$. Then the Galois group $G$ acts transitively on this set. Moreover, let $\sigma\in G$, if $\sigma(\fP_i)=\fP_j$, then $\sigma G_{\fP_i}\sigma^{-1}=G_{\fP_j}$ and $\sigma I_{\fP_i}\sigma^{-1}=I_{\fP_j}$. In particular, if $P$ is unramified, then $\sigma (\fP_i, L/K)\sigma^{-1}=(\fP_j, L/K)$.  
\end{proposition}

\begin{definition}\label{Def: Frob_P}
	With all the above setups, if $P$ is unramified, we define the \emph{Frobenius class} (or simply \emph{Frobenius}) of $P$ to be the conjugacy class of $(\fP, L/K)$ for any one (and hence for all) $\fP|P$. We denote the Frobenius class of $P$ by $(P, L/K)$. 
\end{definition}

Like the Chebotarev Density Theorem in the algebraic number field cases, we have the following results. 

\begin{theorem}\cite[Theorem~9.13A, Chebotarev Density Theorem, first version]{Rosen2002}\label{Thm: Chebotarev_A}
	Let $L/K$ be a Galois extension with $G=\Gal(L/K)$. Let $C\subseteq G$ be a conjugacy class in $G$ and let $S_K'$ be the set of unramified primes of $L/K$. Then 
	$$
	\delta(\set{P\in S_K': (P, L/K)=C})=\frac{\#C}{\#G}.
	$$
\end{theorem}

\begin{theorem}\cite[Theorem~9.13B, Chebotarev Density Theorem, second version]{Rosen2002}\label{Thm: Chebotarev_B}
	Let $L/K$ be a geometric Galois extension with $G=\Gal(L/K)$. Let $C\subseteq G$ be a conjugacy class in $G$ and let $S_K'$ be the set of unramified primes of $L/K$. Suppose the common constant field of $L$ and $K$ has $q$ elements. Then 
	$$
	\#\set{P\in S_K': \deg_K P=N, (P, L/K)=C}=\frac{\#C}{\#G}\frac{q^N}{N}+O\left(\frac{q^{N/2}}{N}\right).
	$$
\end{theorem}

\section{Duality between prime factors of divisors}\label{Sect: duality}

In this section, we show  a duality  between the prime factors of divisors of $K$. For a divisor $D\in \cD^+$, let $d^+(D):=\max\set{\deg P: P|D}$ be the largest degree of the prime factors of $D$, and $d^+(0)=0$ if $D=0$. Let $$Q_S(D):=\#\set{P\in S: \deg P=d^+(D), P|D}$$ be the number of prime factors of $D$ in $S$ attaining the maximal degree $d^+(D)$. For two divisors $A,B\in \cD^+$, we say $A\ge B$ (resp.~$A>B$) if $A-B$ is effective (resp.~effective and nonzero). Then we have an identity between $d_{-}(A)$ and $d^+(A)$ that is analogous to \eqref{duality1}.  Similarly, one may formalize a dual identity that is analogous to  \eqref{duality2}. But for the proof later, we only study the analogue of \eqref{duality1}.

\begin{lemma}[Duality Lemma]\label{duality}
	Suppose $f:\N\to\C$ is an arithmetic function with $f(0)=0$. Then for any effective divisor $A$,  the following identity holds
	\begin{equation}\label{dualityeq}
		\sum_{A\ge B}\mu(B)1_{\mathfrak{D}(K, S)}(B)f(d_{-}(B))=-Q_S(A)f(d^+(A)).
	\end{equation}
	Here $1_{\mathfrak{D}(K,S)}$ is the indicator function on $\mathfrak{D}(K,S)$, which is defined as in \eqref{Definition of distingushable}.
\end{lemma}
\begin{proof} Here we follow the same spirits of the proof of \cite[Lemma 3.3]{SweetingWoo2019}. Observe that if $A=0$, then \eqref{dualityeq} automatically holds. Hence in the following we assume that $A>0$. Suppose we have the prime factorization $A=a_1P_1+\cdots a_rP_r$ with $\deg P_1\le\cdots\le\deg P_r$. Preserving ordering, we split $A$ into $A=A_1+\cdots+ A_k$, where all prime factors of $A_j$ have the same degree. Notice that for any $B\in\cD^+$, we have $1_{\mathfrak{D}(K,S)}(B)=1_{\mathfrak{D}(K,S)}(B')$, where $B'$ is the sum of all the prime factors of $B$ of the smallest degree $d_{-}(B)$. Note that $0\notin\mathfrak{D}(K, S)$ by the definition of $\mathfrak{D}(K, S)$ as in \eqref{Definition of distingushable}. Then
	$$\sum_{A\ge B}\mu(B)1_{\mathfrak{D}(K,S)}(B)f(d_{-}(B))=\sum_{j=1}^k\sum_{A_j\ge B'>0}\mu(B')1_{\mathfrak{D}(K,S)}(B')f(d_{-}(B')) \sum_{A_{j+1}+\cdots+A_k\ge B''}\mu(B'').$$
	
	It follows from the principle of inclusion-exclusion that
	\begin{equation}\label{inclusion-
			exclusion}
		\sum_{A\ge B}\mu(B)=\begin{cases}
			1, & \text{if } A=0,\\
			0, & \text{otherwise}.
		\end{cases}	
	\end{equation}
	It follows that only when $A_{j+1}+\cdots+A_k=0$, i.e. $A_j=A_k$, we have nonvanished term $\sum\limits_{A_{j+1}+\cdots+A_k\ge B''}\mu(B'')$. Thus we have 
	$$\sum_{A\ge B}\mu(B)1_{\mathfrak{D}(K,S)}(B)f(d_{-}(B))=\sum_{A_k\ge B'>0}\mu(B')1_{\mathfrak{D}(K,S)}(B')f(d_{-}(B')).$$
	Moreover, notice that under the condition $A_k\geq B'>0$, $\mu(B')1_{\mathfrak{D}(K,S)}(B')f(d_{-}(B'))\neq 0$ if and only if $B'$ is a prime and $B'\in \mathfrak{D}(K,S)$. In this case noting that $d_{-}(B')=d^+(A_k)=d^+(A)$ and using the definition of $Q_{S}$, we get 
	
	$$\sum_{A\ge B}\mu(B)1_{\mathfrak{D}(K,S)}(B)f(d_{-}(B))=\sum_{A_k\ge B'>0}\mu(B')1_{\mathfrak{D}(K,S)}(B')f(d_{-}(B'))=-Q_S(A)f(d^+(A)),$$
	which completes the proof.
\end{proof}

\section{Asymptotic estimate for the average of \texorpdfstring{$Q_S(A)$}{}}\label{Sec:largest_prime_divisor}

In this section, we mainly show an asymptotic formula for the average of $Q_S(A)$. The main technique for this is the estimation of the number of  $m$-smooth divisors of degree $n$, which is an analogue of the smooth numbers in $\N$. Similarly as Sect.~\ref{Sect: duality}, we only consider effective divisors in this section. For an effective divisor $D\in \cD^+$, we say $D$ is $m$-smooth if $\deg P\le m$ for all $P|D$. Let $$\cS(n,m):=\set{D\in \cD^+: \deg D=n, d^+(D)\le m}$$ be the set of all effective $m$-smooth divisors of degree $n$, and let $\Psi(n,m):=\#\cS(n,m)$. To find an estimate of $\Psi(n,m)$,  by Corollary~\ref{Cor: bn_for_large_n}  we see that 
\begin{equation}\label{averagecounting}
	\sum_{\deg A=n}1=b_n=c_Kq^n+O(1),
\end{equation}
where $c_K={h_Kq^{1-g_{K}}}/{(q-1)}>0$. Then by \cite[Theorem 2]{Manstavichyus1992}, for $1\le u=n/m\ll m/\log m$ and $m\ge m_0\ge3$ we have 
\begin{equation}\label{Psiasymeq}
	\Psi(n,m)=c_Kq^n\rho(u)\of{1+O\of{\frac{1+u\log u}{m}}},
\end{equation}
where $\rho(u)$ is the Dickman function  defined by the unique continuous solution of the  difference-differential equation $u\rho'(u)=-\rho(u-1)$ for $u>1$ with the initial condition $\rho(u)=1$ for $0\le u\le 1$.
Using $\rho(u)< 1/\Gamma(u+1)$,  we get the following estimate  from (\ref{Psiasymeq}).
\begin{lemma}\label{smoothbd}
	For $m\gg \sqrt{n\log n}$, we have that
	\begin{equation}
		\Psi(n,m)\ll q^n\exp\of{-\frac{n}{m}\log\of{\frac{n}{2m}}},
	\end{equation}	
	where the implied constant depends only on $K$.
\end{lemma}
From now on, and until the end of this section, we will assume  $m\le n/10$ and use  $\Psi(n,m)\ll q^n\exp\of{-{n}/{m}}$ for simplicity when applying Lemma~\ref{smoothbd}. Before estimating the average of $Q_S(A)$, we write $\sum_{\deg A=n}Q_S(A)$ in terms of $\Psi(n,m)$ as follows:
\begin{align}\label{qsfphi}
	\sum_{\deg A=n}Q_S(A)&= \sum_{\substack{\deg (P+B)= n\\ P\in S,d^+(B)\le\deg P}}1\nonumber\\
	&=\sum_{\substack{\deg P\le n\\ P\in S}}\sum_{\substack{\deg B=n-\deg P\\ d^+(B)\le\deg P}}1\nonumber\\
	&=\sum_{\substack{\deg P\le n\\ P\in S}}\Psi\of{n-\deg P, \deg P}.
\end{align}

For $\Psi\of{n-\deg P, \deg P}$, we have the following property.

\begin{lemma}\label{inequality_for_Psi}
	Let  $1\le r\le n$ be integers. Then for any $1\le m\le n$, we have the following inequality
	\begin{equation}\label{inequality_for_Psi_eq}
		\sum_{\deg P\le m}\Psi\of{n-r\deg P,\deg P}\le \frac{n}{r}\Psi(n,m).
	\end{equation}
\end{lemma}
\begin{proof}
	Let $\cP_k$ be the set of all prime divisors of degree $k$. Then we define a map $T$ as follows
	\begin{align*}
		T:\bigcup_{k=0}^m\cS(n-rk, k)\times \cP_k&\to \cS(n,m), \\(A,P)&\mapsto A+rP.
	\end{align*}
	For any $D\in \cS(n,m)$, suppose $T((A_1,P_1))=T((A_2,P_2))=D$ for some two pairs $(A_1,P_1)$  and $(A_2,P_2)$. Then $D=A_1+rP_1=A_2+rP_2$. Since $d^+(A_i)\le\deg P_i$ for $i=1,2$, we get that $d^+(D)=\deg P_1=\deg P_2$. It follows that $D$ has at most $n/r$ distinct prime factors of the largest degree $d^+(D)$ whose coefficients are at least $r$. This means that there are at most $n/r$ pairs of the same image under $T$. Therefore, (\ref{inequality_for_Psi_eq}) holds.
\end{proof}

First, we show an asymptotic estimate on the average of $Q_\cP(A)$.
\begin{lemma} \label{average_of_QSA}
	We have
	\begin{equation}\label{average_of_QSA_eq}
		\sum_{ \deg A=n}Q_\cP(A)=c_Kq^n+O_A\of{q^n\exp\of{-c\sqrt{\frac{n}{\log n}}}}.
	\end{equation}
	Here and thereafter mentioned c is a positive constant that may vary from one line to
	the next.  
\end{lemma}
\begin{proof}
	First, we can write the sum as follows
	\begin{align}
		\sum_{ \deg A=n}Q_\cP(A)&=\sum_{ \deg A=n}1+\sum_{ \deg A=n}\of{Q_\cP(A)-1}\nonumber\\
		&=c_Kq^n+\sum_{\substack{\deg A=n\\Q_\cP(A)\ge2}}\of{Q_\cP(A)-1}+O(1).\label{aqfpfeq1}
	\end{align}
	By the similar  argument in  (\ref{qsfphi}), one can see that
	\begin{align}
		\sum_{ \substack{\deg A=n\\Q_\cP(A)\ge2}}\of{Q_\cP(A)-1}&\le\sum_{ \substack{\deg A=n\\Q_\cP(A)\ge2}}Q_\cP(A)=\sum_{\deg P\le n}\Psi\of{n-2\deg P, \deg P}\nonumber\\
		&=\sum_{\deg P<m}\Psi\of{n-2\deg P, \deg P}+\sum_{m\le\deg P\le n}\Psi\of{n-2\deg P, \deg P}\nonumber\\
		&:=S_1+S_2,
	\end{align}
	where $m\gg \sqrt{n\log n}$ is to be chosen later.

	For $S_1$, by Lemma~\ref{inequality_for_Psi}  and Lemma~\ref{smoothbd}, we have
	\begin{equation}
		S_1\ll n\Psi(n,m)\ll nq^{n}\exp\of{-n/m}.
	\end{equation}

	For $S_2$, recall that by the Prime Number Theorem for function fields \eqref{Lem: prime_num_thm_eq} we get $\pi_K(k)\ll q^k/k$. By \eqref{averagecounting}, we have  $\Psi(n-2k,k)\ll q^{n-2k}$. Hence 
	\begin{align}
		\label{aqfpfeq2}
		S_2		&=\sum_{k=m}^{n}\pi_K(k)\Psi(n-2k,k)\nonumber\\
		&\ll \sum_{k=m}^n\frac{q^{k}}{k}\cdot q^{n-2k}\nonumber\\
		&\ll \frac{q^{n-m}}{m}.
	\end{align}
	
	Taking $m=[\sqrt{n\log n}]$ and	combining \eqref{aqfpfeq1} - \eqref{aqfpfeq2}, we deduce the desired asymptotic estimate \eqref{average_of_QSA_eq}.
\end{proof}

Now, we show the following asymptotic estimate for  $Q_S(A)$.

\begin{theorem} \label{equidistribution} 
	For any fixed subset $S\subseteq\cP$ with natural density $\delta(S)$, we have
	\begin{equation}\label{equidistributioneq}
		\sum_{ \deg A=n}Q_S(A)=\delta(S)c_Kq^n+o(q^{n}).
	\end{equation}
\end{theorem}

\begin{proof}
	First, similar to the proof of Lemma~\ref{average_of_QSA}, we break (\ref{qsfphi}) up into two parts:
	\begin{align}
		\sum_{\deg A=n}Q_S(A)
		&=\sum_{\substack{\deg P<m\\ P\in S}}\Psi\of{n-\deg P, \deg P}+\sum_{\substack{m\le \deg P\le n\\ P\in S}}\Psi\of{n-\deg P, \deg P}\nonumber\\
		&:=S_3+S_4,
	\end{align}	
	where $m\gg \sqrt{n\log n}$ is to be chosen later.
	
	For $S_3$, by Lemmas \ref{smoothbd} and \ref{inequality_for_Psi}  again, we have
	\begin{equation}
		S_3\ll \sum_{\deg P<m}\Psi\of{n-\deg P, \deg P} \ll n\Psi(n,m)\ll nq^n\exp\of{-n/m}.
	\end{equation}

	For $S_4$,  we have
	\begin{align}
		S_5:&=S_4-\delta(S)\sum_{m\le \deg P\le n}\Psi\of{n-\deg P, \deg P}\nonumber\\
		&=\sum_{\substack{m\le \deg P\le n\\ P\in S}}\sum_{\substack{\deg B=n-\deg P\\ d^+(B)\le\deg P}}1-\delta(S)\sum_{m\le \deg P\le n}\sum_{\substack{\deg B=n-\deg P\\ d^+(B)\le\deg P}}1\nonumber\\
		&=\sum_{\substack{\deg B\le n-m\\d^+(B)\le n-\deg B}} \sum_{\substack{\deg P=n-\deg B\\P\in S}} 1-\delta(S)\sum_{\substack{\deg B\le n-m\\d^+(B)\le n-\deg B}}  \sum_{\deg P=n-\deg B}1\nonumber\\
		&=\sum_{\substack{\deg B\le n-m\\d^+(B)\le n-\deg B}} \Big( \pi_{K,S}(n-\deg B)-\delta(S)\pi_K(n-\deg B)\Big)\nonumber\\
		&=\sum_{k=m}^n( \pi_{K,S}(k)-\delta(S)\pi_K(k))\sum_{\substack{\deg B=n-k\\d^+(B)\le k}}1.
	\end{align}
	
	To analyze the growth rate of the error term, we define
	\begin{align*}
		e_{K,S}(n)&:=\sup_{k\le n}|\pi_{K,S}(k)-\delta(S)\pi_K(k)|,\\
		v_{K,S}(n)&:=\sup_{k\ge n}	\frac{e_{K,S}(k)}{k\pi_K(k)}, \quad \forall n\ge1.
	\end{align*}
	Then $e_{K,S}(n)$ is nondecreasing,  while $v_{K,S}(n)$ is nonincreasing, as $n\to\infty$. By the inequalities $\pi_K(k)\ll q^k/k$ and $\sum_{\deg B=n-k}1\ll q^{n-k}$, we have
	\begin{align}
		|S_5|&\le \sum_{k=m}^n| \pi_{K,S}(k)-\delta(S)\pi_K(k)|\sum_{\deg B=n-k}1\nonumber\\
		&\ll q^n\sum_{k=m}^n\frac{e_{K,S}(k)}{q^k}\nonumber\\
		&\ll q^n\sum_{k=m}^n v_{K,S}(k)	\nonumber\\
		&\ll nq^nv_{K,S}(m).
	\end{align}
	
	It follows that
	\begin{equation}\label{pfSeq}
		\sum_{\deg A=n}Q_S(A)=\delta(S)\sum_{m\le \deg P\le n}\Psi\of{n-\deg P, \deg P}+R_{K,S}(n,m),	
	\end{equation}
	where
	$$R_{K,S}(n,m)\ll  nq^nv_{K,S}(m)+nq^n\exp\of{-n/m}. $$
	
	Take $S=\cP$, we have
	\begin{equation}\label{pfPeq}
		\sum_{\deg A=n}Q_\cP(A)=\sum_{m\le \deg P\le n}\Psi\of{n-\deg P, \deg P}+O\of{nq^n\exp\of{-n/m}}.
	\end{equation}
	
	Combining (\ref{pfSeq}), (\ref{pfPeq}) and  (\ref{average_of_QSA_eq}) together, we get that
	\begin{equation}
		\sum_{\deg A=n}Q_S(A)=\delta(S)c_Kq^n+E_{K,S}(n,m),
	\end{equation}
	where 
	\begin{equation}\label{pfeboundeq}
		E_{K,S}(n,m) \ll  nq^n v_{K,S}(m)+nq^n\exp\of{-n/m}+q^n\exp\of{-c\sqrt{\frac{n}{\log n}}}.
	\end{equation}
	
	For the choice of admissible $m$, we will use some properties on the growth rate of the error terms $e_{K,S}(n)$ and $v_{K,S}(n)$, see Lemma~\ref{choiceofm} below.
	By Lemma~\ref{choiceofm}(a), we have  $\lim_{n\to\infty}{e_{K,S}(n)}/{\pi_K(n)}=0$.
	This implies that $\lim_{n\to\infty}nv_{K,S}(n)=0$ due to $nv_{K,S}(n)\le \sup_{k\ge n}e_{K,S}(k)/\pi_K(k)$. Thus, by Lemma~\ref{choiceofm}(b), for $h(n)=\sqrt{n/\log n}$, there exists a sequence $\set{k_n}_{n=1}^\infty$ such that $\lim_{n\to\infty}k_n=\lim_{n\to\infty}n/k_n=\infty$, $n/k_n\ll\sqrt{n/\log n}$, and $\lim_{n\to\infty}nv_{K,S}(k_n)=0$.
	Taking $m=k_n$ in (\ref{pfeboundeq}) gives us that $E_{K,S}(n,k_n)=o\of{q^n}$, as desired.
\end{proof}

\begin{lemma}\label{choiceofm}
	(a) Suppose $f,g: \N\to(1,\infty)$ are two arithmetic functions such that
	$$\lim_{n\to\infty}g(n)=\infty \text{ and } \lim_{n\to\infty}\frac{f(n)}{g(n)}=\delta.$$
	Suppose $g(n)$ is monotonically increasing. 
	Define
	$$e(n;f,g):=\sup_{k\le n}|f(k)-\delta g(k)|.$$
	Then 
	\begin{equation}
		\lim_{n\to\infty}\frac{e(n;f,g)}{g(n)}=0.
	\end{equation}
	
	(b)  Given a  decreasing function $f: \N\to(0,\infty)$  with
	$\lim_{n\to\infty}nf(n)=0$. Suppose $h: \N\to[1,\infty)$ is an arbitrary function such that $\lim_{n\to\infty}h(n)=\infty$ and $\lim_{n\to\infty}n/h(n)=\infty$ .
	Then  there exists a  sequence $\set{k_n}_{n=1}^\infty$ of positive integers satisfying $\lim_{n\to\infty}k_n=\infty$ such that 
	\begin{enumerate}
		\item $\lim_{n\to\infty}n/k_n=\infty$.
		\item $n/k_n\ll h(n)$. 
		\item $\lim_{n\to\infty}n\cdot f(k_n)=0$.
	\end{enumerate}
\end{lemma}
\begin{proof} (a) For any $n$, define 
	$$h_n=\sup\set{m\in\N: g(m)\le g(n)^{1/2}\text{ or } m=1}.$$
	Then $h_n$ is nondecreasing, $h_n<n$, $\lim_{n\to\infty}h_n=\infty$ (use contradiction method), and $\lim_{n\to\infty}g(h_n)/g(n)=0$. Since $\lim_{n\to\infty}f(n)/g(n)=\delta$, we have $f(n)\ll g(n)$. It follows that
	\begin{equation}
		\frac{e(n;f,g)}{g(n)}\ll\frac{g(h_n)}{g(n)}+\sup_{h_n\le k\le n}\frac{|f(k)-\delta g(k)|}{g(n)}.
	\end{equation}
	Since $\frac{1}{g(n)}\le\frac1{g(k)}$ for all $k\le n$, we get that
	\begin{equation}
		\frac{e(n;f,g)}{g(n)}\ll\frac{g(h_n)}{g(n)}+\sup_{h_n\le k\le n}\frac{|f(k)-\delta g(k)|}{g(k)}=\frac{g(h_n)}{g(n)}+\sup_{h_n\le k\le n}\abs{\frac{f(k)}{g(k)}-\delta}.
	\end{equation}
	Thus, 
	$$\lim_{n\to\infty}\frac{e(n;f,g)}{g(n)}=\lim_{n\to\infty}\frac{g(h_n)}{g(n)}+\lim_{n\to\infty}\sup_{h_n\le k\le n}\abs{\frac{f(k)}{g(k)}-\delta}=0.$$
	
	(b) Since $\lim_{n\to\infty}nf(n)=0$, for any $m\ge1$, there exists a minimum positive integer constant $N=N(m)$ such that
	\begin{equation}\label{PPTpflimitdfn}
		n f\left(\left[\frac{n}{m}\right]\right)<\frac1m
	\end{equation}
	for all $n>N(m)$. Then $N(m)$ increases as $m$ increases.  And we have that $\lim_{n\to\infty}\sup\set{m\in\N: N(m)<n}=\infty$. Otherwise, suppose the limit is $M<\infty$. Then there exists some $L$ such that $\sup\set{m\in\N: N(m)<n}=M$ for all $n\ge L$. Then for any $n\ge L$, $M+1\notin \set{m\in\N: N(m)<n}$, which means $N(M+1)\ge n$,  contradiction.
	
	Now, we set
	\begin{equation}
		h_n=\min\left(\left[\sqrt{h(n)}\right], \sup\set{m\in\N: N(m)<n \text{ or }m=1}\right).
	\end{equation}
	Here $[x]:=\sup\set{n\in\N: n\le x}$ is the integral part of $x$.
	Then $h_n\le \sqrt{h(n)}$ and $\lim_{n\to\infty}h_n=\infty$. Moreover, we have $n>N(h_n)$ for all $n> N(1)$. By (\ref{PPTpflimitdfn}),  we have
	$$nf\left(\left[\frac{n}{h_n}\right]\right)<\frac1{h_n}$$
	all $n> N(1)$, which implies that $\lim_{n\to\infty}nf\left(\left[{n}/{h_n}\right]\right)=0$. 
	
	Finally, we take $k_n=[{n}/{h_n}]$. Then $k_n$ satisfies the required properties.
\end{proof}

\section{Partial sums of the M\"obius function}\label{Sec:intermediate_thm}

In this section, we mainly show the analogues of Theorem 5 and Theorem 6 in \cite{Alladi1977} for function fields. The first analogue is an estimate on a partial sum of  the M\"obius function. The second one is an  intermediate result that can be used to convert the estimate of the average of $Q_S(A)$ in (\ref{equidistributioneq}) into the desired formula (\ref{mainthmeq}). We firstly  show the analogue of \cite[Theorem 5]{Alladi1977}.

\begin{lemma}\label{averagemu}
	For any bounded function $f$,	we have
	\begin{equation}\label{averagemueq}
		\sum_{\deg A\le n}\mu(A)1_{\mathfrak{D}(K,S)}(A)f(d_{-}(A))=O_{f,K}\of{\frac{q^n}{\log\log n}}.
	\end{equation}
	
\end{lemma}

\begin{proof} Like \cite[Lemma 4.2]{KuralMcDonaldSah2020}, we break up the sum based on the degree of the minimal prime factor $P_{\min}(A)$ of $A$ for  $A\in \mathfrak{D}(K,S)$. The notation $P_{\min}(A)$ is used only when $A$ is distinguishable.  Then
	\begin{align}\label{pfmusetup}
		\sum_{\deg A\le n}\mu(A)1_{\mathfrak{D}(K,S)}(A)f(d_{-}(A))&=\sum_{\substack{\deg P\le n\\ P\in S}}f(\deg P)\sum_{\substack{\deg A\le n\\ P_{\min}(A)=P}}\mu(A)\nonumber\\
		&=\sum_{\substack{\deg P\le m\\ P\in S}}f(\deg P)\sum_{\substack{\deg A\le n\\
				P_{\min}(A)=P}}\mu(A)\nonumber\\
		&\qquad+ \sum_{\substack{m<\deg P\le n\\ P\in S}}f(\deg P)\sum_{\substack{\deg A\le n\\
				P_{\min}(A)=P}}\mu(A)\nonumber\\
		&:=S_6+S_7.
	\end{align}
	where $m$ is to be chosen later. 
	
	For the inside sum, observe that $\sum_{\substack{\deg A\le n\\P_{\min}(A)=P}}\mu(A)=-\sum_{\substack{\deg A\le n-\deg P\\d_{-}(A)>\deg P}}\mu(A)$. 
	We denote by 
	$$M(n,m):=\sum_{\substack{\deg A\le n\\d_{-}(A)>m}}\mu(A), \quad \Phi(n,m):=\sum_{\substack{\deg A\le n\\d_{-}(A)>m}}1$$
	for $m,n\ge1$. For $M(n,m)$, using the argument in \cite[Theorem 1.2]{FujisawaMinamide2018} as \cite[Lemma 4.2]{KuralMcDonaldSah2020}, one can derive that 
	\begin{equation}\label{uniformbound}
		M(n,m)\ll q^{n}\exp\of{-c\sqrt{n}}\prod_{\deg P\le m}\of{1-{|P|^{-1/2}}}^{-1}
	\end{equation}
	uniformly for $n,m\ge1$, where $c>0$ is a positive constant. For the product over $\deg P\le m$, notice that 
	$$-\sum_{\deg P\le m}\log (1-|P|^{-1/2})\ll \sum_{\deg P\le m}|P|^{-1/2}\ll\sum_{k\le m} q^{k/2}/k
	\ll q^{m/2}/m\ll q^m.$$
	It follows that
	\begin{equation}\label{uniformbd}
		M(n,m)\ll q^{n}\exp\of{-c\sqrt{n}+q^{m}}.
	\end{equation}
	
	Thus, by \eqref{uniformbd}, we have
	\begin{align}\label{s6estimate}
		S_6&=-\sum_{\substack{\deg P\le m\\ P\in S}}f(\deg P)M(n-\deg P, \deg P)\nonumber\\
		&\ll \sum_{\deg P\le m}|M(n-\deg P, \deg P)|\nonumber\\
		&\ll\sum_{\deg P\le m} q^{n-\deg P}\exp\of{-c\sqrt{n-\deg P}+q^{\deg P}}\nonumber\\
		&\ll q^{n}\sum_{1\le k\le m}\frac{q^k}{k}\cdot q^{- k}\exp\of{-c\sqrt{n-k}+q^{k}}\nonumber\\
		&\ll q^n\exp\of{-c\sqrt{n-m}+q^m}/m.
	\end{align}	
	
	For $S_7$, we have
	\begin{equation}\label{s7-1}
		S_7\ll \sum_{m<\deg P\le n}\sum_{\substack{\deg A\le n\\
				P_{\min}(A)=P}}1\le \Phi(n,m).
	\end{equation}
	By the sieve of Eratosthenes (e.g., see \cite[\S5.1]{CojocaruMurty2006}), we have
	\begin{equation}
		\Phi(n,m)=c_Kq^n\prod_{\deg P\le m}\of{1-|P|^{-1}}+O\of{n2^{c_Kq^m}}.
	\end{equation}
	Then by the Prime Number Theorem for function fields, we get that
	\begin{equation}\label{s7-2}
		\Phi(n,m)\ll\frac{q^n}{m}+n2^{c_Kq^m}.
	\end{equation}
	
	Taking  $m=[\log\log n]$ and combining \eqref{pfmusetup}, \eqref{s6estimate}, \eqref{s7-1}, and \eqref{s7-2} together, we can get the desired estimate (\ref{averagemueq}).
\end{proof}

\begin{remark}
	As \cite[Theorem 1.2]{FujisawaMinamide2018}, to get the uniform bound \eqref{uniformbound}, we only need to use a new Perron's inversion formula due to Liu and Ye \cite{LiuYe2007} and the fact that the zeta function  $\zeta_K(s)$ of $K$ has a zero-free region like the Riemann zeta function $\zeta(s)$. But it is well-known that the Riemann Hypothesis holds for $\zeta_K(s)$, which says all the zeros of $\zeta_K(s)$ lie on the line $\Re{s}=1/2$. Therefore, one may expect a better bound for $M(n,m)$ by other methods. 
\end{remark}

Now, analogous to \cite[Theorem 6]{Alladi1977}, we  show
the following  duality between the $d_{-}(A)$ and $d^+(A)$. 

\begin{theorem}\label{keylemma}
	For any bounded arithmetic function $f$ with $f(0)=0$, we have that 
	\begin{equation}
		\sum_{ \deg A=n}Q_S(A)f(d^+(A))\sim\delta(S)c_Kq^n
	\end{equation}	
	if and only if
	\begin{equation}
		-\lim_{n\to\infty}  \sum_{\substack{1\le \deg A\le n\\ A\in \mathfrak{D}(K,S)}}\frac{\mu(A)f(d_{-}(A))}{|A|}=\delta(S).
	\end{equation}
\end{theorem}

\begin{proof}
	Let $g=g_K$ be the genus of $K$. Notice that for $n>2g-2$, there is a constant $c_0$ such that  
	$$
	\sum_{\deg A=n}1=c_Kq^n+c_0.
	$$	
	By Lemma~\ref{duality}, we	have
	\begin{align}
		-\delta(S)c_Kq^n&\sim -\sum_{\deg A=n}Q_S(A)f(d^+(A))\nonumber\\
		&=\sum_{ \deg A=n}\sum_{A\ge B}\mu(B)1_{\mathfrak{D}(K,S)}(B)f(d_{-}(B))\nonumber\\
		&=\sum_{\deg B\le n} \mu(B)1_{\mathfrak{D}(K,S)}(B)f(d_{-}(B)) \sum_{\deg C=n-\deg B}1\nonumber\\
		&=\sum_{\deg B<n-2g+2} \mu(B)1_{\mathfrak{D}(K,S)}(B)f(d_{-}(B)) \left(c_K\frac{q^n}{|B|}+c_0\right)+\nonumber\\
		&\qquad\sum_{n-2g+2\le \deg B\le n} \mu(B)1_{\mathfrak{D}(K,S)}(B)f(d_{-}(B)) \sum_{\deg C=n-\deg B}1\nonumber\\
		&=c_Kq^n \sum_{\deg B<n-2g+2} \frac{\mu(B)1_{\mathfrak{D}(K,S)}(B)f(d_{-}(B))}{|B|}\nonumber\\
		&\qquad+c_0\sum_{\deg B<n-2g+2} \mu(B)1_{\mathfrak{D}(K,S)}(B)f(d_{-}(B))\nonumber\\
		&\qquad +O\of{\sum_{k=n-2g+2}^n\abs{\sum_{\deg B=k}\mu(B)1_{\mathfrak{D}(K,S)}(B)f(d_{-}(B))}}\nonumber\\
		&:=S_8+S_9+S_{10}.
	\end{align}
	By Lemma~\ref{averagemu}, we have $S_9=o\of{q^n}$ and $S_{10}=o\of{q^n}$, and hence Theorem~\ref{keylemma}. 
\end{proof}

\begin{remark}
	For  the rational function field $K=\F_q(x)$,  by Lemma~\ref{duality}, we have the following identity
	\begin{equation}\label{finitefielddual}
		-\sum_{\substack{1\le\deg F\le n\\ F\in \mathfrak{D}(q,S)}} \frac{\mu(F)f(d_{-}(F))}{|F|}=\frac1{q^n}\sum_{\deg F= n}Q_S(F)f(d^+(F))
	\end{equation}
	for any function $f$. 
	In this case, Theorem~\ref{keylemma} follows immediately by (\ref{finitefielddual}) without using Lemma~\ref{averagemu}.
\end{remark}

\noindent\textit{Proof of Theorem~\ref{mainthm}}.
Theorem~\ref{mainthm} follows immediately by combining Theorem~\ref{keylemma} taking $f(n)=1$ for $n\ge1$ and Theorem~\ref{equidistribution}.

\section{Some connections with other areas in number theory}\label{Sect: modular_explanation}
\subsection{Connection with \texorpdfstring{$a_{\fp}$}{}'s of elliptic curves over number field}
In this section, we display a connection of Theorem~\ref{mainthm} with the arithmetical property of elliptic curves over number fields. Recall that an \emph{elliptic curve} $E$ over a number field $\mathcal{K}$ is a smooth genus one curve with at least one point defined over $\mathcal{K}$. Every such curve can be represented by a Weierstrass equation  
\begin{equation}\label{Eqn: Weierstass_form}
	E: Y^2=X^3+aX+b
\end{equation}
with some elements $a$ and $b$ in the ring of algebraic integer ring $\cO_{\cK}$ of $\mathcal{K}$. For the basic geometric and arithmetical properties of elliptic curves over number fields, one is suggested to Silverman's book \cite{Silverman1}. Given an elliptic curve $E$, or equivalently, given a Weierstrass form \eqref{Eqn: Weierstass_form}, we can find its discriminant ideal $\Delta_E$, which is a factor of the principal ideal $-16(4a^3+27b^2)$. For every prime ideal $\fp$ of $\cO_{\cK}$, as long as $\fp\nmid \Delta_{E}$, the corresponding residue curve $\tilde{E}_{\fp}$ is still a smooth genus one curve with at least one rational point defined over the residue field $\F_{\fp}$ of $\fp$. Hence in this case, $\tilde{E}_{\fp}$ is an elliptic curve over the finite field $\F_{\fp}$ \cite[Chapter~\RNum{5}]{Silverman1}. In particular, in this case, we have the \emph{zeta function} of $\tilde{E}_{\fp}$ defined by 
\begin{equation}\label{Eqn: zeta_fun}
	Z(\tilde{E}_{\fp}; T):=\exp\left(\sum_{m=1}^{\infty}\#\tilde{E}_{\fp}(\F_{\fp^m})\frac{T^m}{m}\right),
\end{equation}
where the notation $\F_{\fp^m}$ stands for the unique degree $m$ extension of $\F_{\fp}$. If we assume that $q=\#\F_{\fp}$ is the cardinality of $\F_{\fp}$, then due to the Weil conjecture, \eqref{Eqn: zeta_fun} is a rational function of form 
\begin{equation}\label{Eqn: zeta_fun_rational}
	Z(\tilde{E}_{\fp}; T)=\frac{1-a_{\fp}T+qT^2}{(1-T)(1-qT)},
\end{equation}
where $a_{\fp}\in \Z$ an integer with absolute value at most $2\sqrt{q}$; see \cite[\RNum{5}, Theorem~2.2]{Silverman1}. In particular, one sees that 
\begin{equation}\label{Eqn: num_rational_pts}
	\#\tilde{E}_{\fp}(\F_{\fp})=1-a_{\fp}+q.
\end{equation}

Now for each $\fp$, denote by $\F_{\fp}(\tilde{E}_{\fp})$ the function field corresponding to $\tilde{E}_{\fp}$. By the relationship between the geometric points of a smooth curve and divisors of the function field of the same curve (see \cite[Chapter~\RNum{2}]{Silverman1}), one finds that the set $\tilde{E}_{\fp}(\F_{\fp})$ is bijectively corresponding to the set of degree one effective divisors of $\F_{\fp}(\tilde{E}_{\fp})$. Recall that in our main theorem, let $S=\cP$, then \eqref{mainthmeq} can be written as 
\begin{equation}\label{Eqn: ap_deg>1_eq}
	-\lim_{n\to\infty}  \sum_{\substack{2\le \deg D\le n\\ D\in \mathfrak{D}(\F_{\fp}(\tilde{E}_{\fp}),\cP)}}\frac{\mu(D)}{|D|}-\sum_{\deg P=1}\frac{-1}{q}=\delta(\cP)=1.
\end{equation}

By $\sum_{\deg P=1}1=\#\tilde{E}_{\fp}(\F_{\fp})$ and inserting \eqref{Eqn: num_rational_pts} into \eqref{Eqn: ap_deg>1_eq}, we immediately have the following interesting application to arithmetic geometry. Recall that we say a prime ideal $\fp$ of $\cK$ is \emph{good} for $E$ if $\fp$ is not a prime factor of $\Delta_{E}$. 
\begin{theorem}\label{Thm: ap_and_Alladi}
	Let $E$ be an elliptic curve defined over a number field $\cK$. For every good prime $\fp$ of $E$, let $\F_{\fp}(\tilde{E}_{\fp})$ be the corresponding function field  of the residue curve $\tilde{E}_{\fp}$, with constant field $\F_q$, and let $a_{\fp}$ be such that $\#\tilde{E}_{\fp}(\F_{\fp})=1-a_{\fp}+q$, then we have 
	\begin{equation}\label{Eqn: express_ap}
		a_{\fp}=1-q\cdot \lim_{n\to\infty}  \sum_{\substack{2\le \deg D\le n\\ D\in \mathfrak{D}(\F_{\fp}(\tilde{E}_{\fp}),\cP)}}\frac{\mu(D)}{|D|}.
	\end{equation}
\end{theorem}

In particular, according to the modularity of elliptic curves over $\Q$  (see \cite{Wiles1995, Diamond-GTM-228}), we know that  there  exists a newform $f\in \mathcal{S}_2(\Gamma_0(N))$ with $N$ the conductor of $E$ such that for almost all prime integers $p$, the value $a_p$ in \eqref{Eqn: express_ap} is the $p$-th Fourier coefficient of $f$. Here $\mathcal{S}_2(\Gamma_0(N))$ is the space of cusp forms of weight $2$ and level $N$; see \cite{Diamond-GTM-228} for more details.

In the same manner, if we still use the above notations, but for every good prime $\fp$, we take $S_{\fp}=\tilde{E}_{\fp}(\F_{\fp})$, which is a finite set. Thus on one hand, we know that $\delta(S_{\fp})=0$. On the other hand, we have 
\begin{equation}\label{Eqn: ap_another_deg>1_eq}
	-\lim_{n\to\infty}  \sum_{\substack{2\le \deg D\le n\\ D\in \mathfrak{D}(\F_{\fp}(\tilde{E}_{\fp}),S_{\fp})}}\frac{\mu(D)}{|D|}-\sum_{\deg P=1}\frac{-1}{q}=\delta(S_{\fp})=0.
\end{equation}

Again, by $\sum_{\deg P=1}1=\#\tilde{E}_{\fp}(\F_{\fp})$ and inserting \eqref{Eqn: num_rational_pts} into \eqref{Eqn: ap_another_deg>1_eq},  we get an analogue of Theorem~\ref{Thm: ap_and_Alladi}, which is stated as follow. 

\begin{proposition}
	With the same notations as Theorem~\ref{Thm: ap_and_Alladi}, but taking $S_{\fp}=\tilde{E}_{\fp}(\F_{\fp})$ for every good prime $\fp$ of $E$, we have 
	\begin{equation}
		a_{\fp}=(q+1)-q\cdot \lim_{n\to\infty}  \sum_{\substack{2\le \deg D\le n\\ D\in \mathfrak{D}(\F_{\fp}(\tilde{E}_{\fp}),S_{\fp})}}\frac{\mu(D)}{|D|}.
	\end{equation}
\end{proposition}

Now we denote by $P_{E, \fp}(T)=1-a_{\fp}T+qT^2\in \Z[T]$ the $L$-function of $\tilde{E}_{\fp}$, i.e. the numerator of the zeta-function $Z(\tilde{E}_{\fp}; T)$. It is known that 
$$
P_{E, \fp}(T)=(1-\alpha_{\fp}T)(1-\beta_{\fp}T)
$$
such that 
$$
\#\tilde{E}_{\fp}(\F_{\fp^m})=1-(\alpha_{\fp}^m+\beta_{\fp}^m)+q^m
$$
for every integer $m$. Notice that if we denote by $\tilde{E}_{\fp^m}$ the base change $\tilde{E}_{\fp}$. Then by \eqref{Eqn: num_rational_pts} we have an integer $a_{\fp^m}$ such that  
$
\#\tilde{E}_{\fp}(\F_{\fp^m})=1-a_{\fp^m}+q^m
$. 
Hence $a_{\fp^m}=\alpha_{\fp}^m+\beta_{\fp}^m$ for all positive integers $m$. On the other hand, applying \eqref{Eqn: express_ap} to $\tilde{E}_{\fp^m}$, with $\F_{\fp^m}(\tilde{E}_{\fp})$ replaced with $\F_{\fp^m}(\tilde{E}_{\fp^m})$, one has for every $m\in \N$
\begin{equation}\label{Eqn: express_ap_m}
	a_{\fp^m}=1-q^m\cdot \lim_{n\to\infty}  \sum_{\substack{2\le \deg D\le n\\ D\in \mathfrak{D}(\F_{\fp^m}(\tilde{E}_{\fp^m}),\cP)}}\frac{\mu(D)}{|D|}.
\end{equation}
So we can deduce the relationships between the corresponding Alladi's sum.

\begin{proposition}
	With the above notations, for every $m\in \N$, the corresponding terms 
	$$
	1-q^m\cdot \lim_{n\to\infty}  \sum_{\substack{2\le \deg D\le n\\ D\in \mathfrak{D}(\F_{\fp^m}(\tilde{E}_{\fp^m}),\cP)}}\frac{\mu(D)}{|D|}
	$$
	satisfy the Newton's identities of homogeneous symmetric polynomials. In particular, 
	\begin{equation}\label{Eqn: 1_2_sympower}
		1-q^2\cdot \lim_{n\to\infty}  \sum_{\substack{2\le \deg D\le n\\ D\in \mathfrak{D}(\F_{\fp^2}(\tilde{E}_{\fp^2}),\cP)}}\frac{\mu(D)}{|D|}=\left(1-q\cdot \lim_{n\to\infty}  \sum_{\substack{2\le \deg D\le n\\ D\in \mathfrak{D}(\F_{\fp}(\tilde{E}_{\fp}),\cP)}}\frac{\mu(D)}{|D|}\right)^2-2q.
	\end{equation}
	
\end{proposition}
\begin{proof}
	Let $t_m(x,y):=x^m+y^m$ be the $k$th symmetric power-sum of two variables. Then the set $t_m(\alpha_{\fp}, \beta_{\fp})=a_{\fp^m}$ for every non-negative integer $m$. Thus our conclusion follows directly from the Newton's identities of homogeneous symmetric polynomials. Moreover, when $m=2$, one has $(\alpha_{\fp}+\beta_{\fp})^2=\alpha_{\fp}^2+\beta_{\fp}^2+2\alpha_{\fp}\beta_{\fp}$. Observe that $\alpha_{\fp}\beta_{\fp}=q$, we get \eqref{Eqn: 1_2_sympower} via the relationship in \eqref{Eqn: express_ap_m}.
\end{proof}

\subsection{A distribution on the restricted sums of \texorpdfstring{$\mu(D)/|D|$}{}}
In this subsection, we take $\cK=\Q$. Recall that there is a famous equidistribution property of $\{a_p\}$, which is described by the Sato-Tate conjecture. More precisely, over the rational number field $\Q$, this conjecture (which is known) says the followings:

\emph{Assume $E$ does not have CM. Then the values ${a_p/\sqrt{p}}$ (as the prime $p$ varies) are equidistributed in $[-2,2]$ with resepct to the probablity measure $\sqrt{1-x^2/4}\,dx/\pi$, i.e.
	\begin{equation}\label{Eqn: Sato-Tate}
		\lim_{X\to \infty}\frac{\#\set{p\leq X: \alpha\leq \frac{a_p}{\sqrt{p}}\leq \beta}}{\#\set{p\leq X}}=\int_{\alpha}^{\beta}\frac{1}{\pi}\sqrt{1-\frac{x^2}{4}}\,dx.
\end{equation}}
For more details about the known cases and the generalizations of this  conjecture, see \cite{Taylor2008, ClozelHarrisTaylor2008, HarrisShepherd-BarronTaylor2010, Barnet-LambGeraghtyHarrisTaylor2011}. 
With connection stated in Theorem~\ref{Thm: ap_and_Alladi}, we have the following result, which states the equidistribution property of the Alladi sum. 
\begin{corollary}
	If $E$ does not have CM, then the set
	$$
	\set{\frac{1}{\sqrt{p}}- \sqrt{p}\lim_{n\to\infty}  \sum_{\substack{2\le \deg D\le n\\ D\in \mathfrak{D}(\F_{p}(\tilde{E}_{p}),\cP)}}\frac{\mu(D)}{|D|}}_{p}
	$$
	is equidistributed in the closed interval $[-2,2]$ with respect to the measure $\frac{1}{\pi}\sqrt{1-\frac{x^2}{4}}\,dx$.
\end{corollary}

\subsection{The case of general smooth curves}
More generally, let  $C$ be a smooth curve  defined over a number field $\cK$. For every prime $\fp$ at which $C$ has good reduction $\tilde{C}_{\fp}$. All the setups above are also well-defined in this situation. In particular, we have 
$$
Z(\tilde{C}_{\fp}; T)=\frac{P_{2g}(T)}{(1-T)(1-qT)},
$$
where $P_{2g}(T)=1-a_{\fp, C}T+\cdots+q^gT^{2g}\in \Z[T]$ is a polynomial of degree $2g$, where $g$ is the genus of the curve $C$ (also the genus of the function field $\F_{\fp}(\tilde{C}_{\fp})$). And we also have $\#\tilde{C}_{\fp}(\F_{\fp})=1-a_{\fp, C}+q$. Hence arguing similarly as above, we get the following result. 

\begin{theorem}
	Let $C$ be a smooth projective curve defined over a number field $\cK$. For every good prime $\fp$ of $C$, let $\F_{\fp}(\tilde{C}_{\fp})$ be the corresponding function field  of the residue curve $\tilde{C}_{\fp}$, with constant field $\F_q$, and let $a_{\fp}$ be such that $\#\tilde{C}_{\fp}(\F_{\fp})=1-a_{\fp}+q$, then we have 
	\begin{equation}\label{Eqn: express_ap_C}
		a_{\fp}=1-q\cdot \lim_{n\to\infty}  \sum_{\substack{2\le \deg D\le n\\ D\in \mathfrak{D}(\F_{\fp}(\tilde{C}_{\fp}),\cP)}}\frac{\mu(D)}{|D|}.
	\end{equation}
\end{theorem}

\section*{Acknowledgements}
The authors would like to thank Ning Ma for helpful comments. The author Biao Wang is grateful to the support of the Doctoral Dissertation Fellowship for Fall 2020 from the Department of Mathematics at the State University of New York at Buffalo.


\end{document}